\documentclass{amsart}
\usepackage[utf8]{inputenc}

\usepackage{amsfonts,amsmath}
\usepackage{graphicx}
\usepackage{booktabs}
\usepackage{pinlabel}
\usepackage{enumerate}
\usepackage{xcolor,colortbl}
\usepackage{url}

\newtheorem{defn}{Definition}[section]
\newtheorem{lemma}[defn]{Lemma}

\newtheorem{proposition}[defn]{Proposition}

\newtheorem{maintheorem}{Theorem}

\newtheorem{maincorollary}[maintheorem]{Corollary}
\newtheorem{mainexample}[maintheorem]{Example}

\theoremstyle{remark}
\newtheorem*{claim}{Claim}

\newcommand{\zz}{{\mathbb Z}}
\newcommand{\zm}{\zz/2\zz}
\newcommand{\rr}{{\mathbb R}}

\newcommand{\qq}{{\mathbb Q}}

\newcommand{\ozsvath}{Ozsv\'{a}th}
\newcommand{\szabo}{Szab\'{o}}
\newcommand{\saso}{Sa\v{s}o}
\newcommand{\spin}{\ifmmode{\rm Spin}\else{${\rm spin}$\ }\fi}
\newcommand{\spinc}{\ifmmode{{\rm Spin}^c}\else{${\rm spin}^c$\ }\fi}
\newcommand{\spinct}{\mathfrak t}
\newcommand{\spincs}{\mathfrak s}
\newcommand{\tors}{{\it Tors}}

\newcommand{\CP}{\mathbb{CP}}

\newcommand{\RP}{\mathbb{RP}}

\DeclareMathOperator{\lk}{lk}

\DeclareMathOperator{\Char}{Char}

\newcommand{\TwoF}[1]{\Sigma_2(#1)}
\newcommand{\dbc}[1]{\Sigma_2(#1)}

\begin{document}

\title{Unlinking information from 4-manifolds}
\author{Matthias Nagel}
\address{Fachbereich Mathematik\newline\indent
University of Regensburg\newline\indent
93053 Regensburg, Germany}
\email{matthias.nagel@mathematik.uni-regensburg.de}
\author{Brendan Owens}
\address{School of Mathematics and Statistics \newline\indent 
University of Glasgow \newline\indent 
Glasgow, G12 8QW, United Kingdom}
\email{brendan.owens@glasgow.ac.uk}
\thanks{The first author was supported by Deutsche Forschungsgemeinschaft through
the SFB 1085 at the University of Regensburg.}
\date{\today}

\dedicatory{Dedicated to Tim Cochran.}

\begin{abstract}
We generalise theorems of Cochran-Lickorish and Owens-Strle to the case of
links with more than one component.  This enables the use of linking forms on
double branched covers, Heegaard Floer correction terms, and Donaldson's
diagonalisation theorem to complete the table of unlinking numbers for nonsplit prime
links with crossing number nine or less.
\end{abstract}

\maketitle


\section{Introduction}
\label{sec:intro}

Let $L$ be a link in $S^3$.  The unlinking number of $L$ is the minimum number
of crossing changes to convert a diagram of $L$ to a diagram of the unlink,
where the minimum is taken over all diagrams.  This is the obvious
generalisation to links of the unknotting number, and one should expect that
many of the same methods should apply to compute it for examples.  
There have been several successful applications of Donaldson's diagonalisation theorem and
Heegaard Floer homology to the calculation of unknotting numbers, for example \cite{cl,osu1,unknotting,Gu1,altu1}.  It is an interesting problem,
which we begin to address here, to generalise these techniques to the case of links.
The
systematic study of unlinking number for links with more than one component
seems to have been initiated by Kohn in \cite{kohn}, in which he computed
unlinking numbers for all but 5 prime, nonsplit, 2-component links which admit
diagrams with 9 crossings or less.  In this paper we determine the unlinking
number for these remaining examples and provide a complete table of unlinking
numbers for prime nonsplit links with crossing number at most 9.

The main result of this paper is a generalisation of a theorem of Cochran and
Lickorish \cite{cl}, and of a refinement due to the second author and Strle
\cite{imm}, to the case of links with more than one component.  We choose an
orientation on a given link and consider the sum $\sigma+\eta$ of the classical
link signature and nullity.  This sum is equal to $k-1$ for the $k$-component
unlink; it increases by 2 or stays constant when a positive crossing is
changed, and decreases by 0 or 2 when a negative crossing is changed.  Thus if
$\sigma+\eta$ is less than $k-1$ for a given orientation on a $k$-component
link $L$, then any set of crossing changes converting $L$ to the unlink must
include changing at least $(-\sigma-\eta+k-1)/2$ positive crossings.  We show
that if the number of positive crossings changed does not exceed this minimum
then the double branched cover $Y$ of $L$ bounds a smooth definite 4-manifold
$W$ which constrains the linking form of $Y$, and which leads to
obstructions coming from Donaldson's diagonalisation theorem and Heegaard Floer
theory.  Note that the trace of a homotopy associated to a sequence of crossing
changes from $L$ to the the unlink gives rise to an immersion of $k$ disks in
the 4-ball bounded by $L$, with one double point for each crossing change.

\begin{maintheorem}
\label{thm:mainthm} Let $L$ be an oriented $k$-component link in $S^3$ with
signature $\sigma$ and nullity $\eta$, and suppose that $L$ is the oriented
boundary of a properly normally immersed union of $k$ disks in the 4-ball
with $p$ positive double points and $n$ negative double points. Then
\begin{equation}
\label{eq:pineq}
p\ge\frac{-\sigma-\eta+k-1}2.
\end{equation}
If equality holds in \eqref{eq:pineq}
then the double cover $Y$ of $S^3$ branched along $L$ bounds a smooth negative-definite
4-manifold $W$ with
$$b_2(W)=2n-\sigma=2(p+n)+\eta-k+1,$$
for which the restriction map
$$H^1(W;\qq)\to H^1(Y;\qq)$$
is an isomorphism.  
Moreover $H_2(W;\zz)$ contains $p+n$ pairwise disjoint spherical classes of self-intersection $-2$, whose images in $H_2(W;\zz)/\tors$ span a primitive sublattice.
\end{maintheorem}

We denote by $c^*(L)$ the \emph{4-ball crossing number} of a link $L$ in $S^3$.  This is the minimal number of double points in a properly normally immersed union of disks in $D^4$ bounded by the link.  The conditions imposed by
Theorem \ref{thm:mainthm} on the linking form of $Y$ are sufficient to determine the unlinking number and 4-ball crossing number of four of Kohn's remaining examples.  Recall that a 2-component link in $S^3$ has two Murasugi signatures corresponding to two choices of \emph{quasi-orientation} (orientation up to overall reversal).

\begin{mainexample}
\label{eg:kohn}
Let $L$ be a 2-component link in $S^3$ with $c^*(L)<3$.  Suppose that the double branched cover $Y$ has finite cyclic homology group.
Then one of the statements below holds:
\begin{enumerate}
\item $\det(L)=2t^2$ with $t\in\zz$, or
\item $\det (L)$ is divisible by $4$ and  at least one of the signatures of $L$ is in $\{-1,1\}$, or
\item $\det (L)$ is divisible by $16$.
\end{enumerate}
In particular, the links 
$$9^2_3=L9a30,\, 9^2_{15}=L9a15,\, 9^2_{27}=L9a17,\, 9^2_{31}=L9a2$$ 
have unlinking number and 4-ball crossing number $3$.
\end{mainexample}

We refer in Example \ref{eg:kohn} to links using both Rolfsen and Thistlethwaite names.
A proof that the unlinking number of $L9a30$ is 3 was given by Kanenobu and Kanenobu in \cite{KK}.

The following corollary is obtained by combining Theorem \ref{thm:mainthm} with Donaldson's diagonalisation theorem \cite{don} and a construction of Gordon and Litherland \cite{gl}. 

\begin{maincorollary}
\label{cor:alt}
Let $L$ be an oriented nonsplit alternating link with $k$ components and signature $\sigma$, and suppose that $L$ bounds a properly normally immersed union of $k$ disks in the 4-ball with $p=\frac{-\sigma+k-1}2$ positive double points and $n$ negative double points.
Let $m$ be the rank of the positive-definite Goeritz lattice $\Lambda$ associated to an alternating diagram of $L$.  Then $\Lambda$ admits an embedding in the standard integer lattice $Z=\zz^{m+2n-\sigma}=\zz^{m+2(p+n)-k+1}$, with $p+n$ pairwise orthogonal vectors of square $2$ which span a primitive sublattice in the orthogonal complement of $\Lambda$ in $Z$.
\end{maincorollary}

The final unknown example from Kohn's table, as well as the links in Example \ref{eg:kohn}, may be settled using Corollary \ref{cor:alt}.   An alternative method involves the use of Heegaard Floer correction terms of the double branched cover.  We illustrate the use of both methods in Section \ref{sec:eg}.

\begin{mainexample}
\label{eg:L9a10}
The link $9^2_{36}=L9a10$ has unlinking number and 4-ball crossing number $3$.
\end{mainexample}

Combining Examples \ref{eg:kohn} and \ref{eg:L9a10} with previous results of Kauffman-Taylor \cite{kt}, Kohn \cite{kohn}, Kawauchi \cite{kaw14} and Borodzik-Friedl-Powell \cite{bfp} leads to the complete Table \ref{table:u} of unlinking numbers of nonsplit prime links with crossing number at most 9.  We also show that for all  links in Table \ref{table:u}, the unlinking number is equal to the 4-ball crossing number.  

\vskip2mm \noindent{\bf Acknowledgements.}
We are grateful to Maciej Borodzik, Stefan Friedl and Mark Powell for bringing the unknown values in Kohn's table to our attention in their paper \cite{bfp}. 
The first author thanks the University of Glasgow for its hospitality. 
The second author acknowledges the influence of his earlier joint work with \saso\ Strle, especially \cite{imm}.
We thank Mark Powell for helpful comments on an earlier draft, and the anonymous referee for helpful suggestions.


\section{Proofs of Theorem \ref{thm:mainthm} and Corollary \ref{cor:alt}}

For $L$ a link in $S^3$ we denote the two-fold 
branched cover of $L$ by $\TwoF{L}$.  Let $F$ be a connected
 Seifert surface for $L$.
Following Trotter and Murasugi \cite{trotter,murasugi},
the \emph{signature} $\sigma(L)$ of the link $L$ is the signature of the symmetrised Seifert pairing
\begin{align*}
H_1(F) \times H_1(F) &\rightarrow \zz\\
a, b &\mapsto \lk(a_+,b) + \lk(a, b_+),
\end{align*}
where $\lk(a,b)$ denotes the linking number of $a$ and $b$ in $S^3$ and
$b_+$ denotes the cycle obtained by pushing $b$ in the positive normal direction.
The \emph{nullity} $\eta(L)$ of the link $L$ is the dimension of the nullspace of the symmetric form above.  This is also equal to the first Betti number of $\dbc{L}$ \cite{kt}.  We note that if $-L$ denotes the mirror of $L$ then $\sigma(-L)=-\sigma(L)$ and $\eta(-L)=\eta(L)$.

The following lemma is based on \cite[Proposition 2.1]{cl} and \cite[Theorem 5.1]{st}.
\begin{lemma}
Let $L$ be a link in $S^3$.
\label{lem:crossingchange}
Suppose that $L$ and $L'$ are oriented links in $S^3$ such that $L'$ is obtained from $L$ by changing a single negative crossing.  Then
$$\sigma(L')\pm\eta(L')\in\{\sigma(L)\pm\eta(L),\sigma(L)\pm\eta(L)-2\},$$
where the choice of $\pm$ is consistent in all three instances above.
Similarly if $L'$ is obtained from $L$ by changing a single positive crossing, then
$$\sigma(L')\pm\eta(L')\in\{\sigma(L)\pm\eta(L),\sigma(L)\pm\eta(L)+2\}.$$
In either case,
$$|\eta(L')-\eta(L)|\le1.$$
\end{lemma}
\begin{proof}
We take a diagram for $L$ containing a negative crossing, such that changing
this crossing yields $L'$.  The Seifert algorithm, followed by tubing together components if necessary, gives a connected surface $F$; we may
choose a basis for $H_1(F)$ represented by loops, exactly one of which passes
through the distinguished  crossing.  We see that the matrix of the symmetrised
Seifert pairing in this basis is the same for $L'$ as for $L$ save for one
diagonal entry which is reduced by 2.  This will leave the signature and
nullity of the pairing on a codimension one sublattice unchanged.  The last
eigenvalue may also be unchanged in sign, or may change from positive to
negative or zero, or from zero to negative.  The first assertion of the lemma
follows; the second follows on applying the first to $-L$ and $-L'$, and the
third follows easily by restricting consideration to the nullity.
\end{proof}
A surface with boundary is normally properly immersed in $D^4$ if its boundary (respectively, interior) is contained in the boundary (resp., interior) of the 4-ball and its only singularities are normal double points, locally modelled on two transversely intersecting planes in $\rr^4$.  Double points may be given a sign by taking an arbitrary orientation on the surface and comparing an oriented basis for the tangent plane to one sheet of the surface at the singularity followed by an oriented basis on the other sheet, to the ambient orientation.
\begin{lemma}
\label{lem:ineq}
Let $L$ be an oriented $k$-component link in $S^3$ with  signature $\sigma$ and  nullity $\eta$, and suppose that $L$ bounds a properly normally immersed union of 
$k$ disks with $p$ positive double points and $n$ negative double points.
Then
$$n\ge\frac{\sigma-\eta+k-1}2 \quad\mbox{ and }\quad p\ge\frac{-\sigma-\eta+k-1}2.$$
\end{lemma}
\begin{proof}  
From \cite[Proposition 2.1]{imm} the immersed disks bounded by $L$ may be
isotoped so they are given by a concordance, followed by the trace of a regular
homotopy corresponding to changing $p$ positive and $n$ negative crossings,
followed by a nullconcordance.  Since signature and nullity are
concordance invariants \cite{kt}, the inequalities now follow from the previous lemma 
and the fact that the $k$-component unlink has signature zero and nullity
$k-1$.
\end{proof}

The following proposition will be used to prove Theorem \ref{thm:mainthm}, and
may be used to give another proof of Lemma \ref{lem:ineq}.
\begin{proposition}
\label{prop:cl}
Let $L$ be a $k$-component oriented link in $S^3 = \partial D^4$. Suppose $L$
bounds normally properly immersed disjoint oriented disks in 
$D^4$ with $p$ positive and $n$ negative self-intersections. Then
there exists an oriented smooth $4$-manifold $W$
with boundary $\partial W = \TwoF{L}$ such that
\begin{enumerate}
\item the second Betti number of $W$ satisfies $b_2(W) \leq 2(p+n) + 1-k + \eta(L)$;
\item the manifold $W$ has signature $\sigma(W) = -2n + \sigma(L)$;
\item the inclusion $\partial W \subset W$ induces
an injection $H^1(W; \qq) \hookrightarrow H^1(\partial W; \qq)$, with cokernel dimension equal to the nullity of the intersection form on $H_2(W)$;
\item there exist $p+n$ classes in $H_2(W;\zz)$ represented by pairwise disjoint spheres of self-intersection $-2$.  The images of these in $H_2(W;\zz)/\tors$ span a primitive sublattice; in other words the quotient of $H_2(W;\zz)/\tors$ by this sublattice is torsion-free. 
\end{enumerate}
\end{proposition}

The proof is an adaptation to the case of links of arguments given in \cite{cl,imm}.

\begin{proof}
By forming connected sums with $\overline{\CP^2}$ and resolving the singularities as
in \cite[Lemma 3.4]{cl}, we obtain an embedded surface $\Delta$ in
\begin{align*}
C := D^4 \# \left( \underset{p+n}{\#} \overline{\CP^2} \right)
\end{align*}
whose boundary is the link $L$ in $\partial C$ and each of whose components
is a disk. 
The Euler characteristic of $C$ is $\chi(C) = 1 + p +n$.
The first homology of the complement of $\Delta$ in $C$ with $\zm$ coefficients has a basis consisting of the meridians of
the component disks of $\Delta$.  Thus we may take the double cover $W=\dbc{C,\Delta}$ of $C$
branched along $\Delta$.
By definition we have $\partial W = \TwoF{L}$.
\begin{claim}
The Euler characteristic of $W$ is $\chi(W) = 2(p+n) -k +2$.
\end{claim}
This can be seen as follows: we have the equality
\begin{align*}
\chi(C) = \chi(C \setminus \Delta) + \chi(\Delta) =\chi(C \setminus \Delta) + k.
\end{align*}
Therefore the double cover of $C \setminus \Delta$ has
Euler characteristic $2(\chi(C)-k)$. The manifold $W$ is obtained
by gluing  $\Delta$ back in again. We obtain 
\begin{align*}
\chi(W) = 2( \chi(C) - k) + k = 2(p + n) -k +2.
\end{align*}

\begin{claim}
The groups $H^1(W,\partial W;\mathbb{Q})\cong H_3(W;\mathbb{Q})$ are trivial.
\end{claim}
We show that $H_3(W; \zm)$ vanishes.  By Poincar\'{e} duality, $H^1(C, \partial C; \zm) = 0$.
Then note that for a set $F \subset C$ such that
$\partial C \cup F$ is connected, we get
$$
H^0(C;\zm)\overset{\cong}{\rightarrow}H^0(\partial C \cup F; \zm)\rightarrow H^1(C, \partial C \cup F; \zm)\rightarrow H^1(C; \zm)=0.
$$
This implies that $H^1(C, \partial C \cup \Delta; \zm) = 0$.four-
From \cite[Theorem 1]{Lee95} and
dualising, we obtain the following exact sequence:
$$
0=H^1(C, \partial C; \zm)\rightarrow H^1(W, \partial W; \zm)\rightarrow H^1(C, \partial C \cup \Delta; \zm) = 0.
$$
Therefore $H^1(W, \partial W; \zm) = 0$ holds.
By Poincar\'{e} duality we obtain directly that $H_3(W; \zm) = 0$.

\begin{claim}
The inclusion $\partial W \subset W$ induces an injection $H^1(W;\qq) \hookrightarrow H^1(\partial W; \qq)$, whose cokernel is isomorphic to the nullspace of the intersection form on $H_2(W;\qq)$.
In particular, the inequality $\eta(L)=b_1(\partial W) \geq b_1(W)$ holds.
\end{claim}
The long exact sequence of the pair $(W, \partial W)$ with $\qq$ coefficients gives
\begin{align*}
0\rightarrow H^1(W) \rightarrow H^1(\partial W) \rightarrow H^2(W,\partial W) \rightarrow H^2(W),
\end{align*}
with the last map given in appropriate bases by the matrix of the intersection
pairing.  The claim follows immediately.

Combining these claims, we obtain an upper bound for the second Betti number
$b_2(W)$.  The Euler characteristic of $W$ is $\chi(W) = 1 - b_1(W) + b_2(W)$.
This gives
\begin{align*}
b_2(W) &= b_1(W) - 1 + 2(p + n) -k +2 \\
&= 2(p+n) + 1 - k + b_1(W) \\
&\leq 2(p+n) +1 -k + \eta(L).
\end{align*}

We proceed to calculate the signature $\sigma(W)$.
This is done exactly as in \cite{cl}. For the reader's convenience we
repeat the argument.
\begin{claim}
The signature of $W$ is $\sigma(W) = -2n+\sigma(L)$.
\end{claim}
We pick a connected Seifert surface $F_{-L}$ of the link $-L$ with interior pushed into $-D^4$. 
Glue to obtain a closed $4$-manifold $(\widetilde{C}, F) := (C, \Delta) \cup_{(S^3, L)} (-D^4, F_{-L})$. 
We consider the double cover $\widetilde{W}=\TwoF{\widetilde{C},F}$ of $\widetilde{C}$ branched along $F$.
This manifold has a decomposition $\widetilde{W} = W \cup_{\TwoF{L}} X_L$, where $X_L$ is the double
cover $\TwoF{F_{-L}}$ of $-D^4$ branched along $F_{-L}$. By \cite{viro,kt} the signature of $X_L$ is $-\sigma(L)$, and by Novikov additivity 
\cite{index}, we have
\begin{align*}
\sigma(\widetilde{W}) = \sigma(W) + \sigma(X_L).
\end{align*}
Furthermore the $G$-signature theorem \cite{index,Gsig} allows us to express $\sigma(\TwoF{\widetilde{C},F})$ as
\begin{align*}
\sigma(\widetilde{W}) = 2 \sigma(\widetilde{C}) - \frac{1}{2}\left ( [F] \cdot [F]\right )
\end{align*}
The self-intersection of $F$ is $[F]\cdot [F] = -4p$; this is because each intersection point in the original immersed disks is replaced by two fibres of the disk bundle over $S^2$ which is the punctured $\overline{\CP^2}$ used to resolve the singularity, and these two disks have the same orientation in the case of a positive double point and opposite in the case of a negative double point.  Thus $[F]$ has twice the generator of the second homology of this punctured $\overline{\CP^2}$ as a summand in the case of a positive double point, and zero times this generator otherwise.
  We obtain the equations
\begin{align*}
\sigma(W) - \sigma(L)
& = 2 (\sigma(\widetilde{C}) +p) = 2(\sigma(C) +p) = -2n\\
\Rightarrow \sigma(W) &= -2n + \sigma(L).
\end{align*}
 
The existence of $p+n$ pairwise disjoint spherical classes $x_1,\dots,x_{p+n}$ of square $-2$
follows as in \cite[Theorem 2.2]{imm} from the fact that the double cover of a
disk bundle with Euler number $-1$ branched along two fibres is a disk bundle of
Euler number $-2$. It remains to see that their images in $H_2(W;\zz)/\tors$ span a primitive sublattice. 
Let $\alpha \in H_2(W;\zz)$ and
$$m\alpha=\sum b_i x_{i}+\beta,$$
where $\beta$ is a torsion class and $m$ a nonzero integer. Denote
the induced map in homology of the double branched cover by 
$\pi \colon H_2(W;\zz) \rightarrow H_2(C;\zz)$. The classes
$\pi(x_i)$ are pairwise orthogonal of square $-1$. Thus we obtain
\begin{align*}
\pi(m\alpha) \cdot \pi(x_i) = -b_i
\end{align*}
and therefore $m$ divides each $b_i$.
\end{proof}

\begin{proof}[Proof of Theorem \ref{thm:mainthm}]
Suppose as in the theorem that $L$ is an oriented $k$-component link in $S^3$
with  signature $\sigma$ and  nullity $\eta$, and that $L$ bounds a properly
normally immersed union of $k$ oriented disks in the 4-ball with $p$
positive double points and $n$ negative double points.

By Proposition \ref{prop:cl}
there exists a smooth 4-manifold $W$ bounded by $Y=\dbc{L}$ with  
$$2(p+n)+1-k+\eta\ge b_2(W)\ge-\sigma(W)=2n-\sigma.$$ 
This inequality in fact recovers
\eqref{eq:pineq}, and equality in \eqref{eq:pineq} thus proves that $W$ is
negative-definite.  Proposition \ref{prop:cl} also tells us that the inclusion of $Y$
into $W$ induces an isomorphism $H^1(W; \qq) \cong H^1(Y; \qq)$, and of the
existence of $p+n$ disjoint spheres of self-intersection $-2$ whose images in $H_2(W;\zz)/\tors$ span
a primitive sublattice.
\end{proof}

\begin{proof}[Proof of Corollary \ref{cor:alt}]
Recall that a link diagram $D$ determines two Goeritz matrices $G$ and $G'$
(see \cite{gl} for details and Section \ref{sec:eg} for an example).  It is shown
in \cite{gl} that these are the intersection forms of the double covers of
$D^4$ branched along the pushed-in black surface coming from either of the two
chessboard colourings of the regions of $S^2\setminus D$; denote these
simply-connected 4-manifolds by $X$ and $X'$.  The union $X\cup
-X'=\dbc{S^4,F}$ of these Gordon-Litherland manifolds is equal to the double
cover of $S^4=D^4\cup -D^4$ branched along the surface $F$ obtained by pushing
one of these surfaces into each hemisphere of the 4-sphere.  A local argument
at the crossings shows that if $D$ is alternating then $\pm\dbc{S^4,F}$ is a
connected sum of complex projective planes, with one summand for each crossing (the union of
black and white subsurfaces in the neighbourhood of a crossing gives a
punctured $\RP^2$, whose double branched cover is a punctured $\pm\CP^2$; all
signs are the same if the diagram is alternating).

Thus for an alternating diagram one of the Goeritz matrices is
positive-definite and the other is negative-definite.  This recovers the
well-known fact that the nullity of a nonsplit alternating link is zero, since
the Goeritz matrices present $H_1(\dbc{L};\zz)$.

Now suppose $L$ satisfies the hypotheses of the Corollary and that $X$ is the
positive-definite Gordon-Litherland manifold associated to an alternating
diagram of $L$, with intersection lattice $\Lambda=(H_2(X;\zz),G)$ of rank
$m$.  By Theorem \ref{thm:mainthm},  $\dbc{L}$ bounds a
negative-definite manifold $W$ with $b_2(W)=2n-\sigma$, 
and with a primitive sublattice of $H_2(W;\zz)/\tors$ spanned by $p+n$
pairwise orthogonal classes of square $-2$.  Gluing $X$ to $-W$ by
a diffeomorphism of their boundaries gives a smooth closed positive-definite
manifold $\widetilde{X}$.  

By Donaldson's diagonalisation theorem \cite{don} we know that 
$(H_2(\widetilde{X};\zz)/\tors,Q_{\widetilde{X}})$ is isometric
to the standard integer lattice $Z=\zz^{m+2n-\sigma}$.
Now the Mayer-Vietoris exact sequence 
shows that we have an inclusion
$$\Lambda\oplus (H_2(-W;\zz)/\tors,Q_{-W})\subset Z,$$
as required.
\end{proof}

We also note the following generalisation of \cite[Corollary 3.21, second inequality]{kt} and
\cite[Corollary 3]{kohn}.  The lower bound on the number of double points also
follows easily from Lemma \ref{lem:ineq}.  The case of links with nullity zero
is given in \cite[Corollary 4.3]{kaw14}, and indeed the case of nonzero nullity may also be derived from results in \cite{kaw14}.

\begin{lemma}  
\label{lem:det}
Let $L$ be a $k$-component link $L$ in $S^3$ with nullity $\eta$.  Then
$$c^*(L)\ge k-1-\eta.$$
If $c^*(L)\le k-1$, then $\det L=2^{k-1}c^2$ for some $c\in\zz$.
\end{lemma}

\begin{proof}
From Proposition \ref{prop:cl} there exists a smooth 4-manifold $W$ bounded by $Y=\Sigma_2(L)$ with
$$2c^*(L)-(k-1-\eta)\ge b_2(W)\ge c^*(L).$$
This yields the lower bound on $c^*(L)$.  If the nullity $\eta$ is nonzero then the determinant is zero.
If $c^*(L)=k-1$ and $\eta=0$, we find $b_2(W)=c^*(L)$ and, again from Proposition \ref{prop:cl}, $H_2(W;\zz)$ has a basis of pairwise orthogonal classes with self-intersection $-2$.  It follows that the absolute value of the intersection form on $H_2(W;\zz)$ is $2^{k-1}$.  From the long exact sequence in cohomology of the pair $(W,Y)$ we see that
$$\det L=|H^2(Y;\zz)|=2^{k-1}c^2,$$
where $c$ is the order of the image of the torsion subgroup of $H^2(W;\zz)$ in $H^2(Y;\zz)$.
\end{proof}

It is also useful to generalise another basic observation of Kohn \cite[inequality (1)]{kohn}.  This may be used recursively.
\begin{lemma}
\label{lem:sublinks}
Let $L=A\cup B$ be an oriented link in $S^3$ which is a union of sublinks $A$ and $B$.  Then 
$$u(L)\ge u(A)+u(B)+\left|\lk(A,B)\right|,$$
where $\lk(A,B)$ denotes the linking number between $A$ and $B$, and
$$c^*(L)\ge c^*(A)+c^*(B)+\left|\lk(A,B)\right|.$$
\end{lemma}
\begin{proof}
The first inequality follows since crossing changes involving only a sublink $A$ do not change the link type of $B$ or the linking number between $A$ and $B$, and crossing changes involving both $A$ and $B$ change $\lk(A,B)$ by one and do not change the link type of $A$ or $B$.  The second inequality also uses \cite[Proposition 2.1]{imm} which tells us that immersed disks in $D^4$ bounded by $L$ are isotopic to a concordance followed by crossing changes followed by a nullconcordance, and the fact that linking number is a link concordance invariant.
\end{proof}


\section{Linking forms and Heegaard Floer correction terms}
\label{sec:rhod}
In this section we recall some obstructions to a 3-manifold $Y$ bounding a 4-manifold with a given intersection form, with specific attention to the case that $Y$ is the double branched cover of a link in $S^3$.  For more details on this material see, e.g., \cite{mac}.

If $M$ is an orientable manifold of dimension at most 4, then there is a free and transitive action of $H^2(M;\zz)$ on the set $\spinc(M)$ of \spinc structures on $M$.  
For the rest of this section we let $Y$ be a rational homology 3-sphere, i.e., a closed oriented 3-manifold with $b_1(Y)=0$.  There is then a natural inclusion of the set of spin structures on $Y$ into $\spinc(Y)$, and the spin structures are the fixed points of the involution given by conjugation of \spinc structures.

The $\rho$-invariant of $Y$ (see for example \cite{AbsGrHF}) is the function
$$\rho:\spinc(Y)\to \qq/2\zz$$
defined as follows.  For each $\spinct\in\spinc(Y)$ we choose a \spinc 4-manifold $(X,\spincs)$ with a fixed diffeomorphism $\partial X\cong Y$ taking
$\spincs|_{\partial X}$ to $\spinct$, and define
\begin{equation}
\label{eq:rhodef}
\rho(Y,\spinct)=\frac{c_1(\spincs)^2-\sigma(X)}{4}\pmod2.
\end{equation}
This is well-defined due to Novikov additivity of the signature and the fact that for any closed \spinc 4-manifold $(X,\spincs)$ we have $c_1(\spincs)^2\equiv\sigma(X)\pmod8$.  Also note that this changes sign under orientation-reversal, i.e., $$\rho(-Y,\spinct)=-\rho(Y,\spinct).$$

Given an $n\times n$ matrix $Q$ with integer entries, let
$$\Char(Q)=\{\xi\in\zz^n\,:\,\xi_i\equiv Q_{ii}\pmod2,\,i=1,\dots n\}.$$
Then there is a free and transitive action of $\zz^n$ on $\Char(Q)$ given by $(x,\xi)\mapsto\xi+2x$ and correspondingly a free and transitive action of the group $\zz^n/Q\zz^n$ on $\Char(Q)/2Q\zz^n$.
The following proposition is proved in \cite{mac}, and follows from the fact that Chern classes of \spinc structures on a 4-manifold $X$ are characteristic covectors for the intersection form of $X$, together with careful consideration of the long exact sequence in cohomology of the pair $(X,\partial X)$.  It is also observed in \cite{mac} that the $\rho$-invariant is a quadratic enhancement of the linking form, and that the following obstruction to bounding $Q$ is equivalent to that coming from the linking form of $Y$.

\begin{proposition}
\label{prop:rho}
Suppose a rational homology 3-sphere $Y$ is the boundary of a smooth 4-manifold
$W$ with intersection form represented by an $n\times n$ matrix $Q$, with
$\left|\det Q \right|=\delta$.

Then the order of $H=H^2(Y;\zz)$ is $\delta t^2$, and there exists a subgroup $T<H$ of order $t$ and an affine monomorphism 
$$\phi:\Char(Q)/2Q\zz^n\to\spinc(Y)/T$$
which is equivariant with respect to the involutions coming from multiplication by $-1$ on $\Char(Q)$ and conjugation of \spinc structures, such that
\begin{align}
\dfrac{\xi^TQ^{-1}\xi-\sigma(Q)}{4}&\equiv \rho(Y,\spinct)\pmod2\label{eq:rhocong}
\end{align}
for all $\xi\in\Char(Q)$ and $\spinct\in\phi([\xi])$.
\end{proposition}

In \cite{AbsGrHF}, \ozsvath\ and \szabo\ associate a rational number $d(Y,\spinct)$, called a \emph{correction term}, to each \spinc structure $\spinct$ on $Y$.  They showed that the reduction of $d(Y,\spinct)$ modulo 2 is $\rho(Y,\spinct)$, and also that the correction terms give additional constraints on negative-definite \spinc manifolds bounded by $(Y,\spinct)$.  These constraints give rise to the following refinement of Proposition \ref{prop:rho}, which is proved in \cite[Theorem 2.2]{qhs}.

\begin{proposition}
\label{prop:d}
Suppose a rational homology 3-sphere $Y$ is the boundary of a smooth 4-manifold $W$ with negative-definite intersection form represented by an $n\times n$ matrix $Q$, with $\left|\det Q\right|=\delta$.

Then the order of $H=H^2(Y;\zz)$ is $\delta t^2$, and there exists a subgroup $T<H$ of order $t$, and a $\zz/2\zz$-equivariant affine monomorphism 
$$\phi:\Char(Q)/2Q\zz^n\to\spinc(Y)/T$$
satisfying
\begin{align}
\dfrac{\xi^TQ^{-1}\xi+n}{4}&\le d(Y,\spinct)\label{eq:dineq}\\
\mbox{and }\quad
\dfrac{\xi^TQ^{-1}\xi+n}{4}&\equiv d(Y,\spinct)\pmod2\label{eq:dcong}
\end{align}
for all $\xi\in\Char(Q)$ and $\spinct\in\phi([\xi])$.
\end{proposition}

We note that in Propositions \ref{prop:rho} and \ref{prop:d}, the subgroup $T$ is the image in $H^2(Y;\zz)$ of the torsion subgroup of $H^2(X;\zz)$ under the map induced by restriction to the boundary.
It is often helpful to consider the fixed points of the involutions on $\Char(Q)/2Q\zz^n$ and $\spinc(Y)/T$.  In particular each coset of $\spinc(Y)/T$ which contains a spin structure is such a fixed point, and from \cite[Lemma 2.4]{qhs} these are the only fixed cosets unless there exists an element $\alpha$ in $H^2(Y;\zz)/ T$ of order 4 with $2\alpha\in T$.  In particular if the order of $T$ is odd then each fixed coset in $\spinc(Y)/T$ contains a spin structure.  If $Y=\Sigma_2(L)$ is the double cover of a link in $S^3$ then there is a bijection due to Turaev \cite{turaev} between quasi-orientations on $L$ (orientations up to overall reversal) and spin structures on $Y$.  The following lemma is taken from \cite{lconc} and \cite{qa}. 

\begin{lemma}
\label{lem:dbc}
Let $\spinct$ be the spin structure on $Y=\Sigma_2(L)$ associated to an orientation on a link $L$ in $S^3$, and let $\sigma(L)$ be the signature of $L$ with the same orientation.
Then  
$$\rho(Y,\spinct)\equiv-\sigma(L)/4\pmod2.$$
If $L$ is nonsplit and alternating, or indeed quasi-alternating, then
$$d(Y,\spinct)=-\sigma(L)/4.$$
\end{lemma}
\begin{proof}
Let $X$ be the double cover of $D^4$ branched along a pushed-in Seifert surface $F$ for the oriented link $L$.  By \cite[Proposition 3.3]{lconc} there is a unique spin structure $\spincs$ on $X$, and the restriction of $\spincs$ to $Y$ is $\spinct$.  By \cite[Lemma 1.1]{kt} or \cite{viro} the signature of $X$ is equal to the signature of the oriented link.  The formula for the $\rho$-invariant then follows from \eqref{eq:rhodef} since the first Chern class of the spin structure $\spincs$ is zero.

The equality for quasi-alternating links is proved in \cite{qa}.
\end{proof}


\section{Examples}
\label{sec:eg}
In this section we will complete Kohn's table of unlinking numbers for two-component links with crossing number 9 or less by showing that each of
$$9^2_3=L9a30,\, 9^2_{15}=L9a15,\, 9^2_{27}=L9a17,\, 9^2_{31}=L9a2,\,9^2_{36}=L9a10$$ 
has unlinking number 3.  In fact we will show that none of them bound a pair of normally properly immersed disks in the 4-ball with two double points.   Each of these may be unlinked by changing 3 crossings in their minimal diagrams, as may be readily seen from Figures \ref{fig:alllinks} and \ref{fig:L9a10}.

\subsection{The links $L9a2$, $L9a15$, $L9a17$, $L9a30$}
The following lemma implies the statement in Example \ref{eg:kohn}.

\begin{figure}[htbp]
\begin{center}
\labellist
\small\hair 3pt
\pinlabel $L9a2$ at 210 620
\pinlabel $L9a15$ at 700 620
\pinlabel $L9a17$ at 210 40
\pinlabel $L9a30$ at 700 40
\endlabellist
\includegraphics[width=0.70\textwidth]{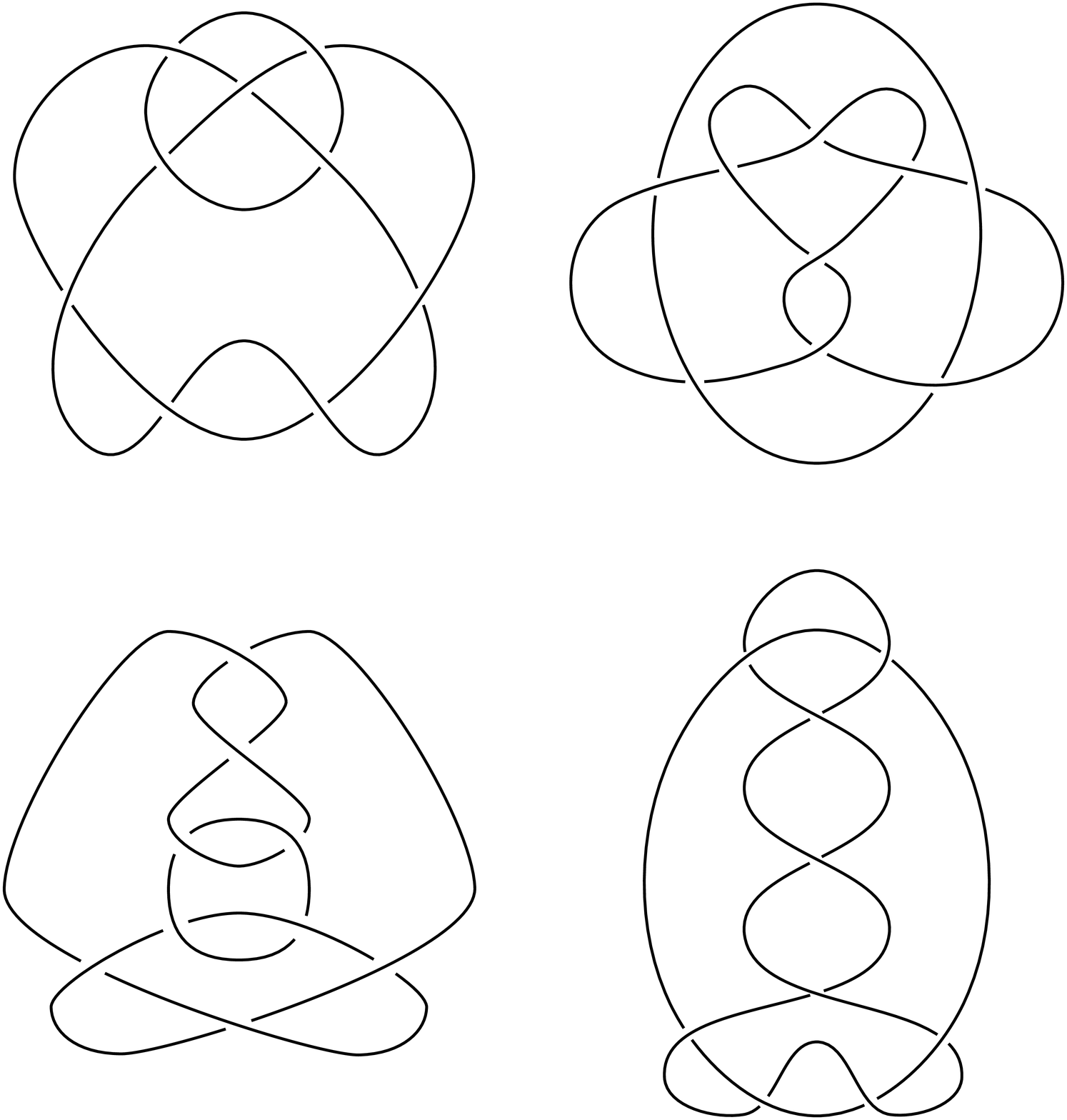}
\caption {The links $L9a2$, $L9a15$, $L9a17$, and $L9a30$.}
\label{fig:alllinks}
\end{center}
\end{figure}

\begin{lemma}
\label{lem:kohn}
For any oriented link in $S^3$, half the absolute value of the signature is a lower bound for the 4-ball crossing number.

Let $L$ be a 2-component link in $S^3$.  Suppose that $H=H^2(\Sigma_2(L);\zz)$ is finite cyclic.  If $c^*(L)<3$, then at least one of the following holds:
\begin{enumerate}[(i)]
\item $\det(L)=2t^2$,  or
\item $\det(L)$ is divisible by 4, and at least one of the signatures of $L$ has absolute value 1, or
\item $\det(L)$ is divisible by 16.
\end{enumerate}
\end{lemma}
\begin{proof}
The well-known fact that the 4-ball crossing number is bounded below by half of the absolute value of
the signature follows as in the proof of Lemma \ref{lem:ineq}, since signature
is a concordance invariant and changes by at most 2 when a crossing changes,
and since the signature of the unlink is zero.

Now suppose $L$ is a 2-component link and $H=H^2(\Sigma_2(L);\zz)$ is finite
cyclic.  Since $H$ is finite, the nullity of $L$ is zero.  It follows from
Lemma \ref{lem:crossingchange} that for any orientation of $L$, the signature
$\sigma$ is odd.  Thus if $c^*(L)<3$ then the signatures of $L$, for either
quasi-orientation, are elements of $S=\{-3,-1,1,3\}$.

By assumption, the link $L$ bounds a normally immersed pair of disks in $D^4$
with $p$ positive and $n$ negative double points, with $p+n=2$ (if $c^*(L)<2$
we  introduce extra double points).
We claim that after possibly reflecting and for some choice of orientation, we have equality in
\eqref{eq:pineq}.  This is immediate if $L$ has signature $-3$, and if $L$ has
signature $3$ then we can reflect.  If $L$ has signature 1, then $n \geq 1$ by Lemma \ref{lem:ineq}.
Thus either $p=0$ and we have equality in \eqref{eq:pineq} for $L$ or $p=1$ and
we have equality in \eqref{eq:pineq} for $-L$.

Assume then that $L$ is oriented with equality in \eqref{eq:pineq}.  By Theorem \ref{thm:mainthm}, $Y=\Sigma_2(L)$ bounds a 4-manifold $W$ whose
intersection form, in some basis, is given by 
$$Q=\begin{pmatrix}-a & b & c\\ b &
-2 & 0\\ c & 0 & -2 \end{pmatrix}.$$ 
Since this is negative-definite and
presents a finite cyclic group, we may assume after change of basis that either
$a=b=1$ and $c=0$, or $b=c=1$ and $a>1$.  In the first case $Q$ has determinant $-2$.  From the long exact sequence of the pair $(W,Y)$
we conclude that $H$ has order $2t^2$ for some integer $t$  (for more details see for example \cite[Lemma 2.1]{qhs}).

Now suppose that the intersection matrix of $W$ is represented by
$$Q=Q_a=\begin{pmatrix}-a & 1 & 1\\ 1 & -2 & 0\\ 1 & 0 & -2
\end{pmatrix},$$
for some $a>0$.  This presents $\zz/(4a-4)\zz$.  It  follows that $H$ has order $(4a-4)t^2$ for some integer $t$; in particular $\det(L)$ is divisible by 4.  If $t$ is even then $\det(L)$ is divisible by 16.  Suppose instead that $t$ is odd.   It follows from \cite[Lemma 2.4]{qhs} that both cosets of  $\spinc(Y)/T$ which are fixed under the
involution coming from conjugation in fact contain a spin structure.  Thus we
can combine Proposition \ref{prop:rho} and Lemma \ref{lem:dbc} to compute the
signatures of $L$ using $Q$.  Fixed points of the involution on
$\Char(Q)/2Q(\zz^n)$ are given by characteristic covectors $\xi$ with
$$\xi=-\xi+2Qx,$$ or in other words $\xi=Qx$ for some $x\in\zz^n$.   There are
two such cosets for $Q_a$: if $a$ is odd they are represented by $(1,-2,0)$ and $(1,0,-2)$, both with square $\xi^TQ_a^{-1}\xi=x^TQ_ax=-2$, and if $a$ is even they are represented by $(0,0,0)$ with square 0 and $(2,-2,-2)$ with square $-4$.  Thus by Proposition \ref{prop:rho} if $a$ is odd then the spin structures on $Y$ both have $\rho$-invariant $(-2+3)/4=1/4$ and if $a$ is even their $\rho$-invariants are $3/4$ and $-1/4$.  In either case we find that at least one of the signatures of $L$ has absolute value 1.
\end{proof}
The double branched cover of the link $L9a30$ has cyclic second cohomology group of order 30.  The links $L9a2$, $L9a15$, and $L9a17$ all have signature $\sigma=-3$ for both choices of orientation, possibly after reflecting, and $H^2(\Sigma_2(L);\zz)\cong\zz/40\zz$.  Applying Lemma \ref{lem:kohn} we conclude that each has $u(L)=c^*(L)=3$. 

\subsection{The link $L9a10$ using Donaldson's diagonalisation theorem}
We take $L$ to be the link $L9a10$, oriented and reflected so that the signature is $1$.

\begin{figure}[htbp]
\begin{center}
\includegraphics[width=\textwidth]{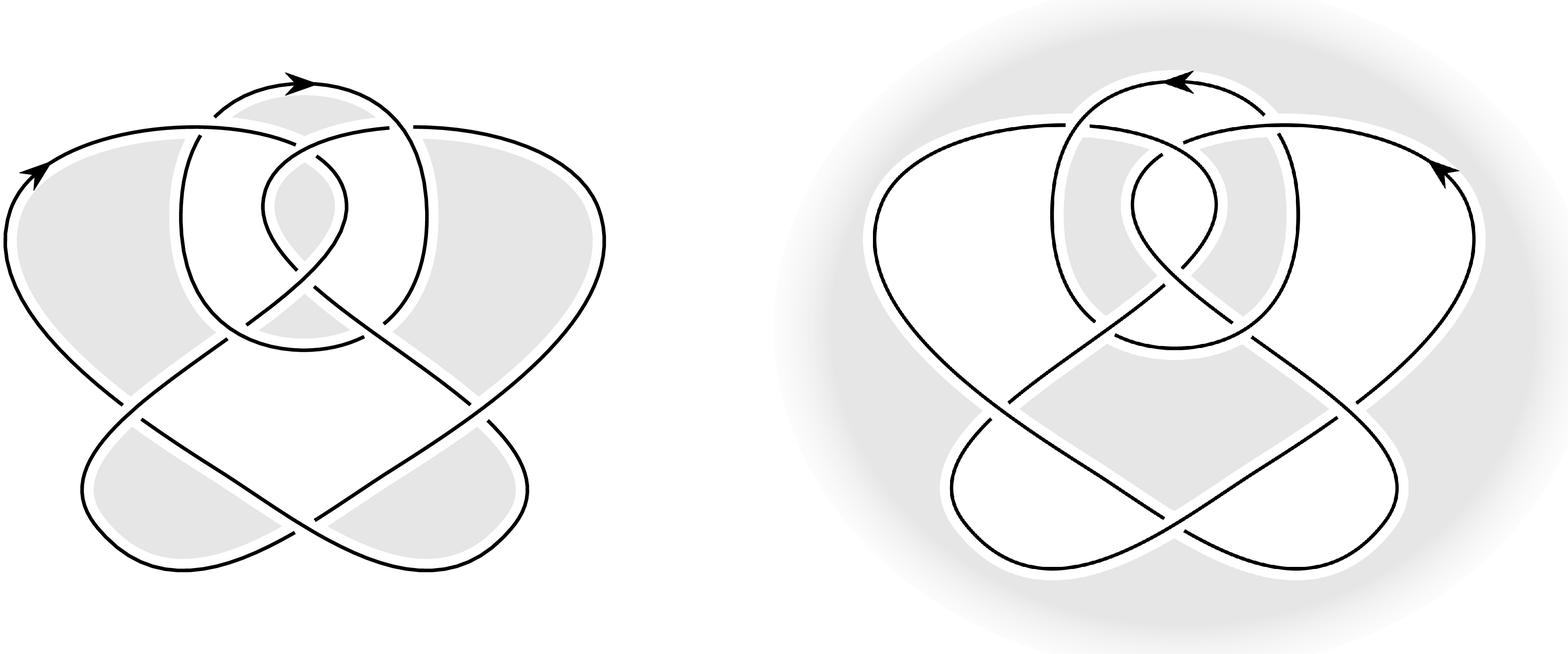}
\caption {The link $L9a10$, with orientation and chessboard colouring.  The link on the left has signature $-1$; its mirror on the right has signature $1$.}
\label{fig:L9a10}
\end{center}
\end{figure}

First consider the link $-L$ as in the first diagram of Figure \ref{fig:L9a10}, which has signature $-1$.  Using the chessboard colouring shown in the figure and following \cite{gl}, the Goeritz matrix is
$$\left(
\begin{matrix} 5&-1&-1\\-1&4&-2\\-1&-2&4 \end{matrix}
\right).$$
Applying \eqref{eq:pineq} to $-L$ leaves two
possibilities: a normally immersed pair of disks in $D^4$ bounded by $-L$ may
have one double point of each sign, or two positive double points.  There are 4 embeddings of the Goeritz lattice $\Lambda$ into the standard integer lattice $\zz^{6}$, 
up to automorphisms of $\Lambda$ and $\zz^6$.  These can be found by hand; they can also be found using the 
\texttt{OrthogonalEmbedings} command in GAP \cite{gap} which implements an algorithm due to Plesken \cite{plesken}.  The complete list is as follows:
$$\begin{aligned}
&\{2e_1+e_2,-e_2+e_3+e_4+e_5,-e_2-e_3-e_4-e_5\},\\
&\{2e_1+e_2,-e_1+e_2+e_3+e_4,-e_2-e_3+e_5+e_6\},\\
&\{e_1+e_2+e_3+e_4+e_5,e_1-e_2-e_3+e_6,-e_1+e_2-e_3-e_6\},\\
&\{e_1+e_2+e_3+e_4+e_5,e_1-e_2-e_3+e_6,-e_1+e_4-e_5-e_6\}.\\
\end{aligned}$$
None of these has a pair of orthogonal square two vectors in the orthogonal complement.  Corollary \ref{cor:alt} then implies we do not have $p=n=1$.  Thus if $-L$ bounds a pair of immersed disks in the 4-ball with two double points then they must both be positive.
  
We now consider $L$, as in the second diagram of Figure \ref{fig:L9a10}, with signature $1$.  Reflecting, we see from the preceding paragraph that if $L$ bounds a pair of immersed disks in the 4-ball with two double points then both are negative, so that we again have equality in \eqref{eq:pineq}.  The Goeritz matrix is
$$\left(
\begin{matrix} 2&-1&0&0&-1&0\\-1&2&-1&0&0&0\\0&-1&3&-1&0&0\\
0&0&-1&3&-1&-1\\-1&0&0&-1&3&0\\0&0&0&-1&0&2 \end{matrix}
\right)$$
and the complete list of embeddings in $\zz^{9}$ is
$$\begin{aligned}
&\{e_1+e_2,-e_2+e_3,-e_3+e_4+e_5,-e_5+e_6+e_7,-e_1-e_4+e_5,-e_7+e_8\},\\
&\{e_1+e_2,-e_2+e_3,-e_3+e_4+e_5,-e_5+e_6+e_7,-e_1-e_7+e_8,-e_4+e_5\},\\
&\{e_1+e_2,-e_2+e_3,-e_3+e_4+e_5,-e_5+e_6+e_7,-e_1-e_7+e_8,-e_6+e_9\},\\
&\{e_1+e_2,-e_2+e_3,-e_3+e_4+e_5,-e_5+e_6+e_7,-e_1-e_7+e_8,-e_7-e_8\},\\
&\{e_1+e_2,-e_2+e_3,-e_1+e_2+e_4,-e_4+e_5+e_6,-e_2-e_3+e_4,-e_5+e_7\}.\\
\end{aligned}$$
Each of the first four do not admit a pair of orthogonal square two vectors in the orthogonal complement.  The last embedding is immediately seen to have image in $\zz^7$, with orthogonal complement in $\zz^9$ isomorphic to
$$\langle3\rangle\oplus\zz^2,$$
which does contain a pair of orthogonal square two vectors.  However these span a finite-index sublattice of $\zz^2$ which is not primitive.  We conclude from Corollary \ref{cor:alt} that $L$ does not bound a pair of immersed disks in $D^4$ with two double points, and that $u(L)=3$.

\subsection{The link $L9a10$ using Heegaard Floer correction terms}
The determinant of $L=L9a10$ is 48, and the double branched cover has cyclic homology group.  As in the argument above using Donaldson's theorem, we apply Theorem \ref{thm:mainthm} to both $-L$, which has signature $-1$ for either quasi-orientation, and to $L$ which has both signatures equal to $1$.
We begin by arguing as in the proof of Lemma \ref{lem:kohn}.   We need to obstruct both $Y=\Sigma(L)$ and $-Y$ from bounding a 4-manifold $W$ with intersection form represented by the matrix
$$Q_a=\begin{pmatrix}-a & 1 & 1\\ 1 & -2 & 0\\ 1 & 0 & -2
\end{pmatrix}.$$
This has determinant $4a-4$.  If it is bounded by $\pm Y$ then $4a-4$ divides $48$ with quotient a square from which it follows that $a$ is either 4 or 13.  The $\rho$-invariants of the spin structures of $\pm Y$ are both $\mp1/4$ by Lemma \ref{lem:dbc}.  This obstructs $-Y$ from bounding $Q_4$ and it obstructs $Y$ from bounding $Q_{13}$,  as in the proof of Lemma \ref{lem:kohn}: if a rational homology sphere $M$ with cyclic homology group bounds $Q_{4}$ then at least one of the spin structures of $M$ has $\rho$-invariant in $\{-1/4,3/4\}$, and if $M$ bounds $Q_{13}$ then at least one of its spin structures has $\rho$-invariant $1/4$.

We complete the proof by  using $\rho$-invariants to show that $-Y$ cannot bound $Q_{13}$, and by using $d$-invariants to show that $Y$ cannot bound $Q_4$.  To do this we need to compute the $d$-invariants of $Y$, which also give the $\rho$-invariants by reducing modulo 2.
The entire set of correction terms of the double branched cover of a nonsplit alternating link can be obtained, using a computer, from the Goeritz matrix of the link as in \cite{osu1}.  We briefly recall how to do this.

Given a negative-definite symmetric integer matrix $G=(g_{ij})$, we  partition the set of characteristic covectors $\xi=(\xi_1,\dots,\xi_n)$ with $g_{ii}\le \xi_i<-g_{ii}$ into cosets of $\Char(G)/2G\zz^n$, and use the Smith normal form of $G$ to record the affine group structure of $\Char(G)/2G\zz^n$.  We also maximise the quantity 
$(\xi^TQ^{-1}\xi+n)/4$ on each of these finite sets of coset representatives; this gives a function 
$$m_G:\Char(G)/2G\zz^n\to\qq.$$
Proposition \ref{prop:d} can then be restated with the left hand side of inequality \eqref{eq:dineq}  replaced by $m_G([\xi])$.  By \cite[Proposition 3.2]{osu1}, if $G$ is the negative-definite Goeritz matrix of a nonsplit alternating diagram of $L$, then in fact $m_G$ computes the correction terms of $\Sigma_2(L)$, or in other words, there is an \emph{isomorphism} $\phi$ as in Proposition \ref{prop:d} with equality in \eqref{eq:dineq}.

In our case we 
find the correction terms of $Y$ are, in cyclic order starting at a spin structure,
$$\left\{
\arraycolsep=1.4pt\def\arraystretch{2.2}
\begin{array}{rrrrrrrrrrrr}
-\dfrac14,&\dfrac{17}{48},&\dfrac{1}{6},&-\dfrac{13}{16},&-\dfrac{7}{12},&-\dfrac{55}{48},&-\dfrac{1}{2},&-\dfrac{31}{48},&\dfrac{5}{12},&\dfrac{11}{16},&\dfrac{1}{6},&-\dfrac{55}{48},\\
-\dfrac{5}{4},&-\dfrac{7}{48},&\dfrac{1}{6},&-\dfrac{5}{16},
&\dfrac{5}{12},&\dfrac{17}{48},&-\dfrac{1}{2},&-\dfrac{7}{48},&-\dfrac{7}{12},&\dfrac{3}{16},&\dfrac{1}{6},&-\dfrac{31}{48},\\
-\dfrac{1}{4},&-\dfrac{31}{48},&\dfrac{1}{6},&\dfrac{3}{16},&-\dfrac{7}{12},&-\dfrac{7}{48},&-\dfrac{1}{2},&\dfrac{17}{48},&\dfrac{5}{12},&-\dfrac{5}{16},
&\dfrac{1}{6},&-\dfrac{7}{48},\\
-\dfrac{5}{4},&-\dfrac{55}{48},&\dfrac{1}{6},&\dfrac{11}{16},&\dfrac{5}{12},&-\dfrac{31}{48},&-\dfrac{1}{2},&-\dfrac{55}{48},&-\dfrac{7}{12},&-\dfrac{13}{16},&\dfrac{1}{6},&\dfrac{17}{48}
\end{array}
\right\}.$$

If we compute the full $\rho$-invariant of $Q=Q_{13}$, for example by computing $m_Q$ as above (though in fact to compute the $\rho$-invariant one just needs an arbitrary representative from each coset of $\Char(Q)/2Q\zz^n$), we find that for $\xi=(3,2,0)$ we have $m_Q([\xi])=-1/12$.  Since the reduction modulo 2 of the $d$-invariants gives the $\rho$-invariants of $Y$, we see that $1/12$ is not the $\rho$-invariant of any \spinc structure on $Y$.  By Proposition \ref{prop:rho} we conclude that $-Y$ does not bound $Q_{13}$.

Now consider the form $Q=Q_4$.  Computing as above we find the values of $m_Q$ are, in cyclic order starting at a fixed point of the involution on $\Char(Q)/2Q\zz^3$,

$$\left\{
\arraycolsep=1.4pt\def\arraystretch{2.2}
\begin{array}{rrrrrrrrrrrr}
&-\dfrac{1}{4},&\dfrac{1}{6},&-\dfrac{7}{12},&-\dfrac{1}{2},&\dfrac{5}{12},&\dfrac{1}{6},\dfrac34,&\dfrac{1}{6},&\dfrac{5}{12},&-\dfrac{1}{2},&-\dfrac{7}{12},&\dfrac{1}{6}
\end{array}
\right\}.$$
There is a unique affine monomorphism $\phi$ from $\Char(Q)/2Q\zz^3$ to $\spinc(Y)/(\zz/2\zz)$ satisfying the conditions of Proposition \ref{prop:rho}.  Labelling the second list above from $0$ to $11$ and the list of $d$-invariants from $0$ to $47$, the only possibility is
$$\phi:i\mapsto \{2i,2i+24\pmod{48}\}.$$
This however (just!) fails to satisfy \eqref{eq:dineq} since $3/4>-5/4$.

We conclude that in fact $Y$ does not bound $Q_4$, and that $u(L)=c^*(L)=3$.


\section{Unlinking numbers of prime links with 9 crossings or less}
Table \ref{table:u} contains the unlinking numbers and 4-ball crossing numbers of all prime nonsplit links with crossing number 9 or less and with more than one component. Unknotting numbers and 4-ball crossing numbers of prime knots with 9 or fewer crossings are available from \cite{knotinfo}.

With the exception of the links in Examples \ref{eg:kohn} and \ref{eg:L9a10}, the unlinking numbers were known previously due to work of Kauffman-Taylor, Kohn, Kawauchi, and Borodzik-Friedl-Powell \cite{kt, kohn, kaw14,bfp}, and in many cases these were also known to be equal to the 4-ball crossing number.

For all links in the table, the unlinking number is realised by changing a subset of the crossings in the minimal diagram given in \cite{knotatlas}.  That the minimal number of crossings in each case is equal to the 4-ball crossing number follows from applying one or more of the results in this paper.

The required lower bound is provided by Lemma \ref{lem:det} for the links $L5a1$, $L6a4$, $L7a1$, $L7a3$, $L7a4$, $L7a6$, $L8a1$, $L8a8$, $L8a9$, $L8a16$, and $L9aN$ for 
$$N\in \{1,3,4,8,9,18,20,21,22,25,26,27,35,38,40,42\}.$$
The bound comes from Lemma \ref{lem:ineq} for $L9a14$, $L9a29$ and $L9a36$, and from Lemma \ref{lem:kohn} for $L9a2$, $L9a15$, $L9a17$, and $L9a30$.  The bound for $L9a10$ is obtained in Section \ref{sec:eg}.  For all remaining links in the table the required bound may be obtained using Lemma \ref{lem:sublinks}.

The fact that the invariants $c^*(L)$ and $u(L)$ coincide in this table is due to the small crossing number restriction.  There are nontrivial nonsplit prime slice links with 
10-crossing diagrams (for example $L10n32$ and $L10n36$ in \cite{linkinfo,knotatlas}) for which these invariants will clearly differ.

\begin{table}[ht]
\begin{center}
\small
\begin{tabular}[t]{lc}
\toprule
Link $L$ & $u(L)$\\
\midrule
$L2a1$	&1\\ 
$L4a1$	&2\\  
$L5a1$	&1\\ 
$L6a1$     &2\\ 
$L6a2$	&3\\ 
$L6a3$	&3\\ 
$L6a4$	&2\\ 
$L6a5$	&3\\ 
$L6n1$	&3\\ 
$L7a1$	&2\\ 
$L7a2$	&3\\ 
$L7a3$	&2\\ 
$L7a4$	&2\\ 
$L7a5$	&1\\ 
$L7a6$	&2\\ 
$L7a7$	&3\\ 
$L7n1$	&3\\ 
$L7n2$	&1\\ 
$L8a1$	&2\\ 
$L8a2$	&1\\ 
$L8a3$	&3\\ 
$L8a4$	&1\\ 
$L8a5$	&3\\ 
$L8a6$	&2\\ 
$L8a7$	&3\\ 
$L8a8$	&2\\ 
$L8a9$	&2\\ 
$L8a10$	&3\\ 
$L8a11$	&3\\ 
$L8a12$	&4\\ 
$L8a13$	&4\\ 
$L8a14$	&4\\ 
$L8a15$	&3\\ 
\bottomrule
\end{tabular}%
\qquad
\begin{tabular}[t]{lcc}
\toprule
Link $L$ & $u(L)$\\
\midrule
$L8a16$	&3\\ 
$L8a17$	&4\\ 
$L8a18$	&4\\ 
$L8a19$	&2\\ 
$L8a20$	&4\\ 
$L8a21$	&4\\ 
$L8n1$	&3\\ 
$L8n2$	&1\\ 
$L8n3$	&4\\ 
$L8n4$	&4\\ 
$L8n5$	&2\\ 
$L8n6$	&4\\ 
$L8n7$	&4\\ 
$L8n8$	&4\\ 
$L9a1$	&2\\ 
$L9a2$	&3\\ 
$L9a3$	&2\\ 
$L9a4$	&2\\ 
$L9a5$	&3\\ 
$L9a6$	&4\\ 
$L9a7$	&3\\ 
$L9a8$	&2\\ 
$L9a9$	&2\\ 
$L9a10$	&3\\ 
$L9a11$	&3\\ 
$L9a12$	&4\\ 
$L9a13$   &3\\ 
$L9a14$	&3\\ 
$L9a15$	&3\\ 
$L9a16$	&3\\ 
$L9a17$	&3\\ 
$L9a18$	&2\\ 
$L9a19$	&2\\ 
\bottomrule
\end{tabular}%
\qquad
\begin{tabular}[t]{lcc}
\toprule
Link $L$ & $u(L)$\\
\midrule
$L9a20$	&2\\ 
$L9a21$	&1\\ 
$L9a22$	&2\\ 
$L9a23$	&4\\ 
$L9a24$	&2\\ 
$L9a25$	&2\\ 
$L9a26$	&2\\ 
$L9a27$	&1\\ 
$L9a28$	&4\\ 
$L9a29$	&3\\ 
$L9a30$	&3\\ 
$L9a31$	&2\\ 
$L9a32$	&4\\ 
$L9a33$	&3\\ 
$L9a34$	&2\\ 
$L9a35$	&2\\ 
$L9a36$	&3\\ 
$L9a37$	&2\\ 
$L9a38$	&1\\ 
$L9a39$	&2\\ 
$L9a40$	&2\\ 
$L9a41$	&2\\ 
$L9a42$	&2\\ 
$L9a43$	&4\\ 
$L9a44$	&4\\ 
$L9a45$	&3\\ 
$L9a46$	&2\\ 
$L9a47$   &3\\ 
$L9a48$	&4\\ 
$L9a49$	&4\\ 
$L9a50$	&3\\ 
$L9a51$	&4\\ 
$L9a52$	&3\\ 
\bottomrule
\end{tabular}%
\qquad
\begin{tabular}[t]{lcc}
\toprule
Link $L$ & $u(L)$\\
\midrule
$L9a53$	&2\\ 
$L9a54$	&3\\ 
$L9a55$	&4\\ 
$L9n1$	&3\\ 
$L9n2$	&2\\ 
$L9n3$	&1\\ 
$L9n4$	&4\\ 
$L9n5$	&2\\ 
$L9n6$	&2\\ 
$L9n7$	&3\\ 
$L9n8$	&2\\ 
$L9n9$	&3\\ 
$L9n10$	&2\\ 
$L9n11$	&2\\ 
$L9n12$	&4\\ 
$L9n13$	&1\\ 
$L9n14$	&2\\ 
$L9n15$	&4\\ 
$L9n16$	&4\\ 
$L9n17$	&2\\ 
$L9n18$	&4\\ 
$L9n19$	&4\\ 
$L9n20$	&4\\ 
$L9n21$	&4\\ 
$L9n22$	&4\\ 
$L9n23$	&3\\ 
$L9n24$	&3\\ 
$L9n25$	&2\\ 
$L9n26$	&3\\ 
$L9n27$	&1\\ 
$L9n28$	&3\\
&\\
&\\
\bottomrule
\end{tabular}
\end{center}
\medskip

\caption{Unlinking numbers and 4-ball crossing numbers of prime nonsplit links with 9 or fewer crossings.}
\label{table:u}
\end{table}


\clearpage
\bibliographystyle{amsplain}
\bibliography{unlinking}

\providecommand{\bysame}{\leavevmode\hbox to3em{\hrulefill}\thinspace}
\providecommand{\MR}{\relax\ifhmode\unskip\space\fi MR }
\providecommand{\MRhref}[2]{%
  \href{http://www.ams.org/mathscinet-getitem?mr=#1}{#2}
}
\providecommand{\href}[2]{#2}
\begin{thebibliography}{10}

\bibitem{index}
M.~F. Atiyah and I.~M. Singer, \emph{The index of elliptic operators. {III}},
  Ann. of Math. (2) \textbf{87} (1968), 546--604.

\bibitem{knotatlas}
D.~Bar-Natan, S.~Morrison, et~al., \emph{The {K}not {A}tlas},
  \url{http://katlas.org}.

\bibitem{bfp}
M.~Borodzik, S.~Friedl, and M.~Powell, \emph{Blanchfield forms and {G}ordian
  distance}, accepted for publication by J. Math. Soc. Japan, arXiv:1409.8421,
  09 2014.

\bibitem{knotinfo}
J.~C. Cha and C.~Livingston, \emph{Table of knot invariants},
  \url{http://www.indiana.edu/~knotinfo}.

\bibitem{linkinfo}
\bysame, \emph{Table of link invariants},
  \url{http://www.indiana.edu/~linkinfo}.

\bibitem{cl}
T.~Cochran and W.~Lickorish, \emph{Unknotting information from
  {$4$}-manifolds}, Trans. Amer. Math. Soc. \textbf{297} (1986), no.~1,
  125--142.

\bibitem{lconc}
A.~Donald and B.~Owens, \emph{Concordance groups of links}, Algebraic and
  {G}eometric {T}opology \textbf{12} (2012), 2069--2093.

\bibitem{don}
S.~K. Donaldson, \emph{The orientation of {Y}ang-{M}ills moduli spaces and
  {$4$}-manifold topology}, J. Differential Geom. \textbf{26} (1987), no.~3,
  397--428.

\bibitem{Gsig}
C.~McA. Gordon, \emph{On the {$G$}-signature theorem in dimension four}, \`{A}
  la recherche de la topologie perdue, Progr. Math., vol.~62, Birkh{\"a}user
  Boston, Boston, MA, 1986, pp.~159--180.

\bibitem{gl}
C.~McA. Gordon and R.~A. Litherland, \emph{On the signature of a link}, Invent.
  Math. \textbf{47} (1978), no.~1, 53--69.

\bibitem{Gu1}
Joshua~Evan Greene, \emph{Donaldson's theorem, {H}eegaard {F}loer homology, and
  knots with unknotting number one}, Adv. Math. \textbf{255} (2014), 672--705.

\bibitem{gap}
{\relax The}~GAP Group, \emph{{GAP} -- {G}roups, {A}lgorithms, and
  {P}rogramming, version 4.7.6}, \url{http://www.gap-system.org}, 2014.

\bibitem{KK}
S.~Kanenobu and T.~Kanenobu, \emph{Oriented {G}ordian distance of two-component
  links with up to seven crossingsordian distance of two-component links with
  up to seven crossings}, Journal of Knot Theory and Its Ramifications
  \textbf{24} (2015).

\bibitem{kt}
L.~Kauffman and L.~Taylor, \emph{Signature of links}, Trans. Amer. Math. Soc.
  \textbf{216} (1976), 351--365.

\bibitem{kaw14}
A.~Kawauchi, \emph{The {A}lexander polynomials of immersed concordant links},
  Bol. Soc. Mat. Mex. (3) \textbf{20} (2014), no.~2, 559--578.

\bibitem{kohn}
P.~Kohn, \emph{Unlinking two component links}, Osaka J. Math. \textbf{30}
  (1993), no.~4, 741--752.

\bibitem{Lee95}
R.~Lee and S.~Weintraub, \emph{On the homology of double branched covers},
  Proc. Amer. Math. Soc. \textbf{123} (1995), no.~4, 1263--1266.

\bibitem{qa}
P.~Lisca and B.~Owens, \emph{Signatures, {H}eegaard {F}loer correction terms
  and quasi--alternating links}, Proc. Amer. Math. Soc. \textbf{143} (2015),
  no.~2, 907--914.

\bibitem{altu1}
Duncan McCoy, \emph{Alternating knots with unknotting number one},
  arXiv:1312.1278, 12 2013.

\bibitem{murasugi}
K.~Murasugi, \emph{On a certain numerical invariant of link types}, Trans.
  Amer. Math. Soc. \textbf{117} (1965), 387--422.

\bibitem{unknotting}
B.~Owens, \emph{Unknotting information from {H}eegaard {F}loer homology}, Adv.
  Math. \textbf{217} (2008), no.~5, 2353--2376.

\bibitem{mac}
B.~Owens and S.~Strle, \emph{Definite manifolds bounded by rational homology
  three spheres}, Geometry and topology of manifolds, Fields Inst. Commun.,
  vol.~47, Amer. Math. Soc., Providence, RI, 2005, pp.~243--252.

\bibitem{qhs}
\bysame, \emph{Rational homology spheres and the four-ball genus of knots},
  Adv. Math. \textbf{200} (2006), no.~1, 196--216.

\bibitem{imm}
\bysame, \emph{Immersed disks, slicing numbers and concordance unknotting
  numbers}, accepted for publication by Comm. Anal. Geom., arXiv:1311.6702, 11
  2013.

\bibitem{AbsGrHF}
P.~Ozsv{{\'a}}th and Z.~Szab{{\'o}}, \emph{Absolutely graded {F}loer homologies
  and intersection forms for four-manifolds with boundary}, Adv. Math.
  \textbf{173} (2003), no.~2, 179--261.

\bibitem{osu1}
\bysame, \emph{Knots with unknotting number one and {H}eegaard {F}loer
  homology}, Topology \textbf{44} (2005), no.~4, 705--745.

\bibitem{plesken}
W.~Plesken, \emph{Solving {$XX^{\rm tr}=A$} over the integers}, Linear Algebra
  Appl. \textbf{226/228} (1995), 331--344.

\bibitem{st}
A.~Stoimenow, \emph{Polynomial values, the linking form and unknotting
  numbers}, Math. Res. Lett. \textbf{11} (2004), no.~5-6, 755--769.

\bibitem{trotter}
H.~F. Trotter, \emph{Homology of group systems with applications to knot
  theory}, Ann. of Math. (2) \textbf{76} (1962), 464--498.

\bibitem{turaev}
V.~Turaev, \emph{Classification of oriented {M}ontesinos links via spin
  structures}, Topology and {G}eometry -- {R}ohlin {S}eminar (Oleg~Y. Viro,
  ed.), Lecture Notes in Mathematics, vol. 1346, Springer, 1988, pp.~271--289.

\bibitem{viro}
O.~Ja. Viro, \emph{Branched coverings of manifolds with boundary, and
  invariants of links. {I}}, Izv. Akad. Nauk SSSR Ser. Mat. \textbf{37} (1973),
  1241--1258.

\end{thebibliography}
\end{document}